\documentclass[11pt]{amsart}
\usepackage{amsfonts}
\usepackage{amssymb,latexsym}
\usepackage{amsmath}
\usepackage{amsthm}
\usepackage{amssymb}
\usepackage{enumerate}
\newtheorem{theorem}{Theorem}[section]
\newtheorem{lemma}[theorem]{Lemma}

\newtheorem{proposition}[theorem]{Proposition}

\theoremstyle{definition}

\newtheorem{conjecture}[theorem]{Conjecture}

\theoremstyle{remark}
\newtheorem{remark}[theorem]{Remark}
\numberwithin{equation}{section}

\begin{document}
\title[A fully-nonlinear flow and quermassintegral inequalities]
{A fully-nonlinear flow and quermassintegral inequalities in the sphere}
\dedicatory{Dedicated to Joseph Kohn on the occasion of his 90th birthday}

\author{Chuanqiang Chen}
\address{Chuanqiang Chen, School of Mathematics and Statistics, Ningbo University, Ningbo, 315211, Zhejiang Province, P.R. China}\email{chenchuanqiang@nbu.edu.cn}
\author{Pengfei Guan} \address{Pengfei Guan, Department of Mathematics, McGill University, Montreal, Quebec,\linebreak H3A2K6, Canada}\email{guan@math.mcgill.ca}

\author{Junfang Li} \address{Junfang Li, Department of Mathematics, University of Alabama at Birmingham, Birmingham, AL 35294, USA} \email{jfli@uab.edu}

\author{Julian Scheuer}
\address{Julian Scheuer, School of Mathematics, Cardiff University, Cardiff CF24 4AG, Wales, UK}
\email{scheuerj@cardiff.ac.uk}
\thanks{ Research of CC was supported by NSFC NO. 11771396, research of PG was supported in part by NSERC Discovery Grant, research of JL was supported in part by NSF DMS-1007223, and JS was supported by the "Deutsche Forschungsgemeinschaft" (DFG, German research foundation), Project "Quermassintegral preserving local curvature flows", No. SCHE 1879/3-1.}

\begin{abstract}
This expository paper presents the current knowledge of particular fully nonlinear curvature flows with local forcing term, so-called locally constrained curvature flows. We focus on the spherical ambient space. The flows are designed to preserve a quermassintegral and to de-/increase the other quermassintegrals. The convergence of this flow to a round sphere would settle the full set of quermassintegral inequalities for convex domains of the sphere, but a full proof is still missing. Here we collect what is known and hope to attract  wide attention to this interesting problem.
\end{abstract}



\maketitle

\section{Introduction}

Let $M^{n}$ be a smooth, closed and connected manifold and let $X : M^n \hookrightarrow \mathbb{S}^{n+1}$ be the embedding of a strictly convex hypersurface. Let $p\in \mathbb{S}^{n+1}$ be a point in the interior of the convex body enclosed by $M$, such that $M$ lies in the interior of the hemisphere determined by $p$ and denote by
\begin{align}
ds^2  = d\rho ^2  + \phi (\rho )^2 dz^2
\end{align}
the metric in polar coordinates around $p$, where $\phi (\rho ) = \sin(\rho)$, $\rho \in [0, \frac{\pi}{2} )$, is the radial distance, and $dz^2$ is the induced standard metric on $\mathbb{S}^n$. 

We consider the following locally constrained curvature flow in the sphere:
\begin{align} \label{1.3}
&\frac{{\partial X}}{{\partial t}} = (c_{n,k} \phi '(\rho ) - u\frac{{\sigma _{k + 1} (\lambda )}}{{\sigma _k (\lambda )}})\nu, \\
&X( \cdot,0)=X_0(\cdot). \notag
\end{align}
where $X(x,t) \in \mathbb{S}^{n+1}$ is the position vector of the evolving hypersurface $M(t)$, $\nu$ the outward unit normal,  $u =\langle\phi (\rho )\frac{\partial }{{\partial \rho }}, \nu\rangle$, $\lambda =(\lambda_1, \cdots, \lambda_n )$ the  principal curvatures, $X_0: M_0 \hookrightarrow \mathbb{S}^{n+1}$ the initial embedded hypersurface, $ \sigma_k$ the $k$-th elementary symmetric function,
 and $c_{n,k}  = \frac{{\sigma _{k + 1} (I)}}{{\sigma _k (I)}} = \frac{{n - k}}{{k + 1}}$, $I=(1, \cdots, 1)$.

The particular interest in these flows stems from its monotonicity properties with respect to the quermassintegrals for convex bodies $\Omega$ in the sphere. Let $M=\partial \Omega$, set
\begin{align}
\mathcal{A}_{-1}= & \mathrm{Vol}(\Omega), \quad \mathcal{A}_0  = \int_M {d\mu _g }, \nonumber \\
\mathcal{A}_1  =& \int_M {\sigma _1 (\lambda )d\mu _g }  + n \mathrm{Vol}(\Omega ), \\
\mathcal{A}_m  =& \int_M {\sigma _m (\lambda )d\mu _g }  + \frac{{n - m + 1}}{{m - 1}}\mathcal{A}_{m - 2}, \nonumber
\end{align}
where $2 \leq m \leq n$. Here $g$ is the induced metric on $M$ and $d\mu_{g}$ the associated volume element.

The monotonicity properties of those functionals along the flow \eqref{1.3} follow from
the following Hsiung-Minkowski identities (see Proposition 2.6):
\begin{align}
(m + 1)\int_M {u\sigma _{m + 1} (\lambda )}  = (n - m)\int_M { \phi (\rho )'\sigma _m (\lambda )}, \quad  0 \leq m \leq n-1.
\end{align}
With the help of the evolution equations (see Proposition 2.9):
\begin{align}
\partial _t \mathcal{A}_{-1}  = \int_M {\left(c_{n,k} \phi '(\rho ) - u\frac{{\sigma _{k + 1} (\lambda )}}{{\sigma _k (\lambda )}}\right)d\mu _g },
\end{align}
and
\begin{align}
\partial _t \mathcal{A}_l  = (l + 1)\int_M {\sigma _{l + 1}(\lambda ) \left(c_{n,k} \phi '(\rho ) - u\frac{{\sigma _{k + 1} (\lambda )}}{{\sigma _k (\lambda )}}\right)d\mu _g },
\end{align}
we deduce that along the flow \eqref{1.3} for $0 \leq k \leq n-1$, the following monotonicity relations hold:
\begin{align}
\partial _t \mathcal{A}_l \left\{ \begin{array}{l}
  \ge 0,\quad \text{ if } l < k-1; \\
  = 0,\quad \text{ if } l = k-1; \\
  \le 0,\quad \text{ if } l > k-1. \\
 \end{array} \right.
\end{align}
Hence, if one can prove that the flow \eqref{1.3} moves
an arbitrary convex hypersurface to a round sphere, then the following conjecture
would turn into a theorem:

\begin{conjecture} \label{th1.2}
\begin{align}\label{1.7}
\mathcal{A}_l \leq \xi_{l,k}(\mathcal{A}_k ), \quad \forall -1\le l<k\le n,
\end{align}
where $\xi_{l,k}$ is the unique positive function defined on $(0, \infty)$ such that `` = " holds when
$M$ is a geodesic sphere. `` = " holds if and only if $M$ is a geodesic sphere.
\end{conjecture}

Flow \eqref{1.3} is another example of hypersurface flows which have been introduced recently with goals to
establish optimal geometric inequalities \cite{GL09, BGL18,GL15, GL18, GLW16} for hypersurfaces in space forms. These locally constrained flows are associated to the optimal solutions to the problems of calculus of variations in geometric setting.
The counterpart of \eqref{1.3} in $\mathbb R^{n+1}$ was considered in \cite{GL15, GL18}, where the longtime existence and convergence were proved by transforming the equation to corresponding inverse type PDE on $\mathbb S^n$ for the support function. In the case of $\mathbb S^{n+1}$, up to several special values of $k$ and $l$, this conjecture is open until today, see for example \cite{BGL18, CS21,MS16}. The main issue is that so far we can not control the curvature along the flow \eqref{1.3} from above (except the case $k=0$ \cite{GL15}).
 All the other a priori estimates for this flow are in place and this note is supposed to collect those estimates.

\medskip

\medskip

The rest of this article is organized as follows. In section 2, we list some basic facts for $k$-th elementary symmetric functions, hypersurfaces in $\mathbb{S}^{n+1}$ and evolution equations. In section 3, we prove the $C^0$, $C^1$ a priori estimates and uniform bounds of $F=\frac{{\sigma _{k + 1} (\lambda )}}{{\sigma _k (\lambda )}}$. In section 4, we prove the strict convexity of $M(t)$ along the flow \eqref{1.3} if $M_0$ is convex. In the last section, we give a discussion of the $C^2$ estimate.

\section{Preliminary}

We first recall some well-known facts about $k$-th elementary symmetric functions, hypersurfaces in $\mathbb{S}^{n+1}$, and then give some evolution equations along the flow \eqref{1.3}.

\subsection{Elementary symmetric functions}
For any $k = 1, \cdots, n$, and $\lambda=(\lambda_1, \cdots, \lambda_n)$, the $k$-th elementary symmetric function is defined as follows
\begin{equation}\label{2.1}
 \sigma_k(\lambda) = \sum _{1 \le i_1 < i_2 <\cdots<i_k\leq n}\lambda_{i_1}\lambda_{i_2}\cdots\lambda_{i_k}.
\end{equation}
and $\sigma_0 =1$. First, we denote by $\sigma _k (\lambda \left| i \right.)$ the symmetric
function with $\lambda_i = 0$ and $\sigma _k (\lambda \left| ij
\right.)$ the symmetric function with $\lambda_i =\lambda_j = 0$.
\begin{proposition}\label{prop2.1}
Let $\lambda=(\lambda_1,\dots,\lambda_n)\in\mathbb{R}^n$ and $k
= 1, \cdots,n$, then
\begin{align*}
&\sigma_k(\lambda)=\sigma_k(\lambda|i)+\lambda_i\sigma_{k-1}(\lambda|i), \quad \forall \,1\leq i\leq n,\\
&\sum_i \lambda_i\sigma_{k-1}(\lambda|i)=k\sigma_{k}(\lambda),\\
&\sum_i\sigma_{k}(\lambda|i)=(n-k)\sigma_{k}(\lambda).
\end{align*}
\end{proposition}

Viewing $\sigma_k$ as a function on symmetric matrices, we also denote by $\sigma _k (W \left|
i \right.)$ the symmetric function with $W$ deleting the $i$-row and
$i$-column and $\sigma _k (W \left| ij \right.)$ the symmetric
function with $W$ deleting the $i,j$-rows and $i,j$-columns. Then
we have the following identities.
\begin{proposition}\label{prop2.2}
Suppose $W=(W_{ij})$ is diagonal, and $m$ is a positive integer,
then
\begin{align*}
\frac{{\partial \sigma _m (W)}} {{\partial W_{ij} }} = \begin{cases}
\sigma _{m - 1} (W\left| i \right.), &\text{if } i = j, \\
0, &\text{if } i \ne j.
\end{cases}
\end{align*}
and
\begin{align*}
\frac{{\partial ^2 \sigma _m (W)}} {{\partial W_{ij} \partial W_{kl}
}} =\begin{cases}
\sigma _{m - 2} (W\left| {ik} \right.), &\text{if } i = j,k = l,i \ne k,\\
- \sigma _{m - 2} (W\left| {ik} \right.), &\text{if } i = l,j = k,i \ne j,\\
0, &\text{otherwise }.
\end{cases}
\end{align*}
\end{proposition}

Recall that the G{\aa}rding's cone is defined as
\begin{equation}\label{2.2}
\Gamma_k  = \{ \lambda  \in \mathbb{R}^n :\sigma _i (\lambda ) >
0,\forall 1 \le i \le k\}.
\end{equation}

The following properties are well known. 

\begin{proposition}\label{prop2.3}
Let $\lambda \in \Gamma_k$ and $k \in \{1,2, \cdots, n\}$. Suppose that
$$
\lambda_1 \geq \cdots \geq \lambda_k \geq \cdots \geq \lambda_n,
$$
then we have
\begin{align}
\label{2.3}& \sigma_{k-1} (\lambda|n) \geq \sigma_{k-1} (\lambda|n-1) \geq \cdots \geq \sigma_{k-1} (\lambda|k) \geq \cdots \geq \sigma_{k-1} (\lambda|1) >0; \\
\label{2.4}& \lambda_1 \geq \cdots \geq \lambda_k >  0, \quad \sigma _k (\lambda)\leq C_n^k  \lambda_1 \cdots \lambda_k; \\
\label{2.5}& \sigma _k (\lambda)\geq  \lambda_1 \cdots \lambda_k, \quad \text{ if } \lambda \in \Gamma_{k+1};
\end{align}
where $C_n^k = \frac{n!}{k! (n-k)!}$.
\end{proposition}

The generalized Newton-MacLaurin inequality is as follows, which will be used all the time.
See \cite{S05}.
\begin{proposition}\label{prop2.4}
For $\lambda \in \Gamma_k$ and $k > l \geq 0$, $ r > s \geq 0$, $k \geq r$, $l \geq s$, we have
\begin{align} \label{2.6}
\Bigg[\frac{{\sigma _k (\lambda )}/{C_n^k }}{{\sigma _l (\lambda )}/{C_n^l }}\Bigg]^{\frac{1}{k-l}}
\le \Bigg[\frac{{\sigma _r (\lambda )}/{C_n^r }}{{\sigma _s (\lambda )}/{C_n^s }}\Bigg]^{\frac{1}{r-s}}.
\end{align}
\end{proposition}

\subsection{Hypersurfaces in $\mathbb{S}^{n+1}$}


The following lemma is well known, e.g. \cite{GL15}.
\begin{lemma}\label{lem2.5}
Let $M^n \subset \mathbb{S}^{n+1}$ be a closed hypersurface with induced metric $g$. Let $\Phi=\Phi (\rho ) = \int_0^\rho  \phi (r)dr$, then
\begin{align} \label{2.7}
\nabla _i \nabla _j \Phi  = \phi '(\rho )g_{ij}  - h_{ij}  u,
\end{align}
recall $u=\langle\phi(\rho) \frac{\partial }{\partial \rho}, \nu\rangle$.
\end{lemma}

\medskip

We have the following Hsiung-Minkowski identities, see \cite{GL15}.

\begin{proposition}\label{prop2.6}
Let $M$ be a closed hypersurface in $\mathbb{S}^{n+1}$. Then, for $m =0, 1, \cdots, n-1$,
\begin{align} \label{2.8}
(m + 1)\int_M {u\sigma _{m + 1} (\lambda )}  = (n - m)\int_M { \phi' (\rho )\sigma _m (\lambda )},
\end{align}
where we use the convention that $\sigma_0 = 1$.
\end{proposition}

Next, we state the gradient and hessian of the support function $u=\langle\phi(\rho) \frac{\partial }{\partial \rho}, \nu\rangle$ under the induced
metric $g$ on $M$, see \cite{GL15}.

\begin{lemma}\label{lem2.7}
The support function $u$ satisfies
\begin{align} 
\nabla _i u =& g^{ml} h_{im} \nabla _l \Phi,  \\
\nabla _i \nabla _j u =& g^{ml} \nabla _m h_{ij} \nabla _l \Phi  + \phi 'h_{ij}  - (h^2 )_{ij} u,
\end{align}
where $(h^2 )_{ij}  = g^{ml} h_{im} h_{jl}$.
\end{lemma}

\subsection{Evolution equations}

Let $M(t)$ be a smooth family of closed hypersurfaces in $\mathbb{S}^{n+1}$, and $X(\cdot, t)$ denote a point on $M(t)$. The following basic evolution equations for normal variations are well known, e.g. \cite{G06}.

\begin{proposition} \label{prop2.8}
Under the flow $\partial_t X = f(X(\cdot ,t)) \nu$ in the sphere
we have the following evolution equations
\begin{align}
 \partial _t g_{ij}  =& 2fh_{ij},  \\
 \partial _t d\mu _g  =& f\sigma_1(\lambda)d\mu _g,  \\
 \partial _t h_{ij}  =&  - \nabla _i \nabla _j f + f(h^2 )_{ij}  - fg_{ij},  \\
 \label{EV-W}\partial _t h_j^i  =&  - \nabla ^i \nabla _j f - fg^{im} (h^2 )_{mj}  - f\delta _j^i.
\end{align}

\end{proposition}

From Proposition \ref{prop2.8}, we can obtain the evolution of the quermassintegrals in the sphere.
\begin{proposition}\label{prop2.9}
Along the flow $\partial_t X = f(X(\cdot ,t)) \nu$ in the sphere,
we have for $0 \leq l \leq n-1$
\begin{align} \label{2.15}
\partial _t \mathcal{A}_l  = (l + 1)\int_M {\sigma _{l + 1}(\lambda) f d\mu _g },
\end{align}
and
\begin{align} \label{2.16}
\partial _t \mathcal{A}_{-1} = \int_M {fd\mu _g },
\end{align}
where $\Omega$ is the domain enclosed by the closed hypersurface. Moreover, if the flow is \eqref{1.3} and $M(t)$ is strictly convex, then we have
\begin{align} \label{2.17}
\partial _t \mathcal{A}_l \left\{ \begin{array}{l}
  \ge 0,\quad \text{ if } l < k-1; \\
  = 0,\quad \text{ if } l = k-1; \\
  \le 0,\quad \text{ if } l > k-1. \\
 \end{array} \right.
\end{align}

\end{proposition}
\begin{proof}  (\ref{2.15}) and (\ref{2.16}) follow directly form Proposition \ref{prop2.8}. (\ref{2.17}) follows from (\ref{2.8}), (\ref{2.15}) and the Newton-MacLaurin inequality. \end{proof}

Let $(M, g)$ be a hypersurface in $\mathbb{S}^{n+1}$ with induced metric $g$. We now give the local
expressions of the induced metric, second fundamental form, Weingarten curvatures etc.
when $M$ is a graph of a smooth and positive function $\rho(z)$ on $\mathbb{S}^n$. Let $\partial_1, \cdots, \partial_n$ be a local
frame along $M$ and $\partial_\rho$ be the vector field along the radial direction. Then the support function,
induced metric, inverse metric matrix, second fundamental form can be expressed as follows. For simplicity, all the covariant derivatives with respect to the standard spherical metric $e_{ij}$ will also be denoted as $\nabla$ when there is no confusion in the context.
\begin{align}
 \label{2.18} u =& \frac{{\phi ^2 }}{{\sqrt {\phi ^2  + |\nabla \rho |^2 } }}, \\
 \label{2.19} g_{ij}  =& \phi ^2 e _{ij}  + \rho_{i} \rho_{j},  \\
 \label{2.20} g^{ij}  =& \frac{1}{{\phi ^2 }}(e^{ij}  - \frac{{\rho^i \rho^j }}{{\phi ^2  + |\nabla \rho |^2 }}) ,\\
 \label{2.21} h_{ij}  =& \frac{1}{{\sqrt {\phi ^2  + |\nabla \rho |^2 } }}( - \phi \nabla _i \nabla _j \rho  + 2\phi '\rho _i \rho _j  + \phi ^2 \phi 'e_{ij} ), \\
\label{2.22}  h_j^i  =& \frac{1}{{\phi ^2 \sqrt {\phi ^2  + |\nabla \rho |^2 } }}(e^{im}  - \frac{{\rho^i \rho^{m} }}{{\phi ^2  + |\nabla \rho |^2 }})( - \phi \nabla _m \nabla _j \rho  + 2\phi '\rho _m \rho _j  + \phi ^2 \phi 'e_{mj} ),
\end{align}
where all the covariant derivatives $\nabla$ and $\rho_i$ are w.r.t. the spherical metric $e_{ij}$ .

We now consider the flow equation \eqref{1.3} of radial graphs over $\mathbb{S}^n$ in $\mathbb{S}^{n+1}$. Let $$\omega  = \frac{\phi}{\sqrt {\phi^2 + |\nabla \rho |^2 }}.$$ It is known that
if a family of radial graphs satisfy $\partial _t X  = f \nu$, then the evolution of the scalar function $\rho = \rho(X(z, t), t)$ satisfies
\begin{align} \label{2.29}
\partial _t \rho  = f\omega.
\end{align}

The following is a well known commutator identity.
\begin{lemma}\label{lem2.10}
Let $h_{ij}$ be the second fundamental form and $g_{ij}$ be the induced metric of a
hypersurface in $\mathbb{S}^{n+1}$. Then
 \begin{align} \label{2.31}
 \nabla _i \nabla _j h_{ml}  = \nabla _m \nabla _l h_{ij}  &+ h_{ij} (h^2 )_{ml}  - (h^2 )_{ij} h_{ml}  + h_{il} (h^2 )_{mj}  - (h^2 )_{il} h_{mj} \notag \\
  &+ [h_{ml} g_{ij}  - h_{ij} g_{ml}  + h_{mj} g_{il}  - h_{il} g_{mj} ].
\end{align}
\end{lemma}

\begin{lemma} \label{lem2.11}
Along the flow \eqref{1.3} in $\mathbb{S}^{n+1}$, the graph function $\rho$ and the support function $u =\langle \phi(\rho) \frac{\partial }{\partial \rho}, \nu\rangle$ evolve as follows
\begin{equation}\label{EV-rho} \partial _t \rho - uF^{ij} \nabla_i\nabla_j\rho  =\frac{\phi'}{\phi}u (c_{n,k}-F^{ij}g_{ij})+\frac{\phi'}{\phi}uF^{ij}\rho_{i}\rho_{j} \end{equation}
\begin{align} \label{2.32}
 \partial _t u - uF^{ij} \nabla_i\nabla_j u  =  - c_{n,k} \nabla \Phi \nabla \phi ' + F\nabla \Phi \nabla u + (c_{n,k} \phi ' - 2uF)\phi ' + u^2 F^{ij} (h^2 )_{ij},
 \end{align}
where $F=\frac{{\sigma _{k + 1} (\lambda )}}{{\sigma _k (\lambda )}}$ and $F^{ij}= \frac{\partial F}{\partial h_{ij}}$.
\end{lemma}

\begin{proof} The function $\rho$ satisfies
\[\partial_t\rho=(c_{n,k}\phi'-uF)\omega,\]
and
\[uF^{ij}\rho_{ij}=-u\omega F+\frac{\phi'}{\phi}uF^{ij}g_{ij}-\frac{\phi'}{\phi}uF^{ij}\rho_{i}\rho_{j}.\]
Hence (\ref{EV-rho}) holds.

Applying Lemma \ref{lem2.7}, we can obtain
\begin{eqnarray*} \label{2.33}
\partial _t u - uF^{ij} u_{ij}  &=& f\phi ' - \nabla \Phi \nabla f - uF^{ij} [\nabla h_{ij} \nabla \Phi  + \phi 'h_{ij}  - (h^2 )_{ij} u] \notag \\
&=& (c_{n,k} \phi ' - uF)\phi ' - \nabla \Phi \nabla (c_{n,k} \phi ' - uF)\notag  \\
&& - u[\nabla F\nabla \Phi  + \phi 'F - F^{ij} (h^2 )_{ij} u] \notag \\
&=&  - c_{n,k} \nabla \Phi \nabla \phi ' + F\nabla \Phi \nabla u + (c_{n,k} \phi ' - 2uF)\phi ' + u^2 F^{ij} (h^2 )_{ij}.
\end{eqnarray*}

\end{proof}

\begin{lemma} \label{lem2.12}
Let $h_{ij}$ be the second fundamental form and $g_{ij}$ be the induced metric of a
hypersurface in $\mathbb{S}^{n+1}$ and $F=\frac{{\sigma _{k + 1} (\lambda )}}{{\sigma _k (\lambda )}}$. Then 
\begin{align} \label{2.34}
 \nabla _i \nabla _j F =& F^{\alpha \beta } \nabla_\alpha  \nabla _\beta  h_{ij}  + F^{\alpha \beta ,\gamma \eta }\nabla _i h_{\alpha\beta }  \nabla _j h_{\gamma\eta }  \notag   \\
  &+ [F^{\alpha \beta }  (h^2 )_{\alpha  \beta } - F^{\alpha \alpha }  ] h_{ij}- F [ (h^2 )_{ij} - g_{ij} ],
\end{align}
where $F^{\alpha \beta}= \frac{\partial F}{\partial h_{\alpha \beta}}$ and $F^{\alpha \beta, \gamma\eta}= \frac{\partial^2 F}{\partial h_{\alpha \beta}\partial h_{\gamma\eta}}$.
\end{lemma}

\begin{proof}

\begin{align} \label{2.35}
 \nabla _i \nabla _j F =& F^{\alpha \beta } \nabla _i \nabla _j h_{\alpha\beta }   + F^{\alpha \beta ,\gamma \eta } \nabla _i h_{\alpha\beta }  \nabla _j h_{\gamma\eta }  \notag  \\
  =& F^{\alpha \beta } \nabla_\alpha  \nabla _\beta  h_{ij}  + F^{\alpha \beta ,\gamma \eta }\nabla _i h_{\alpha\beta }  \nabla _j h_{\gamma\eta }  \notag   \\
  &+ F^{\alpha \beta }  [h_{ij} (h^2 )_{\alpha  \beta }  - (h^2 )_{ij} h_{\alpha  \beta }  + h_{i\beta } (h^2 )_{\alpha j}  - (h^2 )_{i\beta } h_{\alpha  j}   \notag \\
  &\qquad \quad + (h_{\alpha  \beta } g_{ij}  - h_{ij} g_{\alpha \beta }  + h_{i\beta } g_{\alpha j}  - h_{\alpha j} g_{i\beta } )]  \notag \\
  =& F^{\alpha \beta } \nabla_\alpha  \nabla _\beta  h_{ij}  + F^{\alpha \beta ,\gamma \eta }\nabla _i h_{\alpha\beta }  \nabla _j h_{\gamma\eta }  \notag   \\
  &+ [F^{\alpha \beta }  (h^2 )_{\alpha  \beta } - F^{\alpha \alpha }  ] h_{ij}- F [ (h^2 )_{ij} - g_{ij} ].
\end{align}

\end{proof}

\begin{lemma} \label{lem2.13}
Let $h_{ij}$ be the second fundamental form and $g_{ij}$ be the induced metric of a
hypersurface in $\mathbb{S}^{n+1}$ and $F=\frac{{\sigma _{k + 1} (\lambda )}}{{\sigma _k (\lambda )}}$. Then along the flow \eqref{1.3}
\begin{align} \label{2.36}
  \partial _t h_j^i  -   u F^{ml} \nabla _m \nabla _l h_{j}^i=& u F^{ml,pq} \nabla^i h_{ml} \nabla _j h_{pq}  + \nabla^i u\nabla _j F + \nabla _j u\nabla^i F + F\nabla h_{j}^i \nabla \Phi  \notag \\
  &- (c_{n,k} \phi ' + uF)(h^2 )_{j}^i  \notag \\
  &+ h_{j}^i [u(F^{ml} (h^2 )_{ml}  - F^{mm} ) + \phi 'F - c_{n,k} u] \notag \\
  &+ 2uF\delta _j^i,
\end{align}
and
\begin{align} \label{2.37}
 \partial _t F - uF^{ml} \nabla _m \nabla _l F =& 2F^{ml} \nabla _m u\nabla _l F + F\nabla \Phi \nabla F  \notag\\
  &- [c_{n,k} F^{ml} (h^2 )_{ml}  - F^2 ]\phi ' + uF[\sum {F^{mm} }  - c_{n,k} ].
\end{align}
\end{lemma}

\begin{proof}
By the tensorial property, we do not distinguish upper and lower indexes in this
proof whenever applicable. We need the fact that $\nabla \phi '=-\nabla \Phi$, and then we can obtain
\begin{align}
\partial _t h_j^i  =&  - \nabla _i \nabla _j (c_{n,k} \phi ' - uF) - (c_{n,k} \phi ' - uF)(h^2 )_{ij}  - (c_{n,k} \phi ' - uF)\delta _j^i  \notag \\
=&  - c_{n,k} \nabla _i \nabla _j \phi ' + u\nabla _i \nabla _j F + \nabla _i u\nabla _j F + \nabla _j u\nabla _i F + F\nabla _i \nabla _j u \notag \\
&- (c_{n,k} \phi ' - uF)(h^2 )_{ij}  - (c_{n,k} \phi ' - uF)\delta _j^i  \notag \\
=& c_{n,k} (\phi 'g_{ij}  - uh_{ij} ) + \nabla _i u\nabla _j F + \nabla _j u\nabla _i F \notag \\
&+ F[\nabla h_{ij} \nabla \Phi  + \phi 'h_{ij}  - u(h^2 )_{ij} ] \notag \\
&+ u[F^{ml} \nabla _m \nabla _l h_{ij}  + F^{ml,pq} \nabla _i h_{ml} \nabla _j h_{pq}  + (F^{ml} (h^2 )_{ml}  - F^{mm} )h_{ij}  - F((h^2 )_{ij}  - g_{ij} )] \notag \\
&- (c_{n,k} \phi ' - uF)(h^2 )_{ij}  - (c_{n,k} \phi ' - uF)\delta _j^i \notag  \\
=& u[F^{ml} \nabla _m \nabla _l h_{ij}  + F^{ml,pq} \nabla _i h_{ml} \nabla _j h_{pq} ] + \nabla _i u\nabla _j F + \nabla _j u\nabla _i F + F\nabla h_{ij} \nabla \Phi \notag  \\
&- (c_{n,k} \phi ' + uF)(h^2 )_{ij}  \notag \\
&+ h_{ij} [u(F^{ml} (h^2 )_{ml}  - F^{mm} ) + \phi 'F - c_{n,k} u] \notag \\
&+ 2uF\delta _j^i.  \notag
\end{align}

Finally, \eqref{2.37} follows from
$$\partial_t F = F^i_j \partial_t h^j_i.$$
\end{proof}

\begin{lemma}  \label{lem2.14}
Let $h_{ij}$ be the second fundamental form and $g_{ij}$ be the induced metric of a
hypersurface in $\mathbb{S}^{n+1}$ and $F=\frac{{\sigma _{k + 1} (\lambda)}}{{\sigma _k (\lambda )}}$. Then along the flow \eqref{1.3}
\begin{align} \label{2.38}
 \partial _t (uF) - uF^{ml} \nabla _m \nabla _l (uF) =& F[c_{n,k} |\nabla \Phi |^2  + (c_{n,k} \phi ' - 2uF)\phi ' + u^2 F^{ml} (h^2 )_{ml} ] \notag \\
&+ u[uF(\sum {F^{mm} }  - c_{n,k} ) - (c_{n,k} F^{ml} (h^2 )_{ml}  - F^2 )\phi ' ]\notag \\
& + F\nabla \Phi \nabla (uF) ,
  \end{align}
and
\begin{align} \label{2.39}
 \partial _t (\frac{{h_i^i }}{u}) - uF^{ml} \nabla _m \nabla _l (\frac{{h_i^i }}{u}) =& F^{ml,pq} \nabla ^i h_{ml} \nabla _i h_{pq}  + \frac{2}{u}\nabla ^i u\nabla _i F + F\nabla \Phi \nabla (\frac{{h_i^i }}{u})\notag\\
  &+ 2F^{ml} \nabla _m u\nabla _l (\frac{{h_i^i }}{u})- (c_{n,k} \frac{{\phi '}}{u} + F)(h^2 )_i^i  + 2F \notag \\
  &+ h_i^i [F^{mm}  - (c_{n,k} \frac{{\phi '}}{{u^2 }} - 3\frac{F}{u})\phi ' - c_{n,k}  + \frac{{c_{n,k} }}{{u^2 }}\nabla \Phi \nabla \phi '].
  \end{align}
\end{lemma}

\section{A priori estimates}

Since $M_0$ is strictly convex, there is $T>0$ such that flow (\ref{1.3}) exists and  the solution $M(t)$ is strictly convex for all $0\le t< T$. {\it This will be assumed in the rest of this section.}

\subsection{$C^0$ estimate}

\begin{theorem} \label{th4.1}
Let $M_0$ be a strictly convex, radial graph of positive function $\rho_0$ over $\mathbb{S}^n$ embedded in  $\mathbb{S}^{n+1}$.
If $M(t)$ solves the flow \eqref{1.3} with the initial value $M_0$, then for any $(z,t) \in \mathbb{S}^n \times [0, T)$
\begin{align}\label{4.1}
\mathop {\min }\limits_{z \in \mathbb{S}^n} \rho_0(z) \leq \rho(z,t) \leq \mathop {\max }\limits_{z \in \mathbb{S}^n} \rho_0(z).
\end{align}
\end{theorem}

\begin{proof}
At critical points of $\rho$, we have the following critical point conditions,
\begin{align}
 \nabla \rho  = 0,\quad  \omega  = 1, \quad u = \phi,
\end{align}
and then the Weingarten curvature is
\begin{align}
h_j^i  = \frac{1}{\phi^2}( - \rho_{ij}  + \phi \phi 'e_{ij} ),
\end{align}
and then
\begin{align}
 F = \frac{{\sigma _{k + 1} (\lambda )}}{{\sigma _k (\lambda )}} = \frac{{\phi '}}{\phi }\frac{\sigma _{k + 1} ( -\frac{\rho_{ij} }{\phi\phi '} + e_{ij} )}{\sigma _k ( - \frac{\rho_{ij} }{\phi\phi '} + e_{ij} )}.
 \end{align}
So at the critical point, we have
\begin{align}
 \partial_t \rho = c_{n,k} \phi' - u F = c_{n,k} \phi' - \phi ' \frac{\sigma _{k + 1} ( - \frac{\rho_{ij} }{\phi\phi '} + e_{ij} )}{\sigma _k ( - \frac{\rho_{ij} }{\phi\phi '} + e_{ij} )}.
 \end{align}

By standard maximum principle, this proves the upper and lower bounds for $\rho$.
\end{proof}

\subsection{$C^1$ estimates}
\begin{lemma}\label{lem4.2}
Let $F(\lambda) := \frac{\sigma_{k+1}(\lambda)}{\sigma_{k}(\lambda)}$, and $c_{n,k} = F(I) = \frac{n-k}{
k+1}$ where $I = (1, \cdots, 1)$. Then
\begin{align}\label{4.6}
\sum_i F^{ii} \lambda_i^2 \geq \frac{F^2}{c_{n,k}}, \quad \sum_i F^{ii} \geq c_{n,k}, \quad \forall \lambda\in \Gamma_k.
\end{align}
Moreover, if $\lambda \in \bar \Gamma_{k+1}$, then $\sum F^{ii} \leq n-k$.
\end{lemma}
\begin{proof} The proof below is from \cite{BGL18}
    We first derive
    \begin{equation}
      \begin{array}[]{rll}
        \sum F^{ii}\lambda_i^2 =&\sum \frac{\sigma_{k}(\lambda|i)\lambda_i^2\sigma_k -\sigma_{k+1}\sigma_{k-1}(\lambda|i)\lambda_i^2}{\sigma_k^2}\\
\\
=&\frac{[\sigma_1\sigma_{k+1}-(k+2)\sigma_{k+2}]\sigma_k -\sigma_{k+1}[\sigma_1\sigma_k-(k+1)\sigma_{k+1}]}{\sigma_k^2}\\
\\
=&\frac{(k+1)\sigma^2_{k+1}-(k+2)\sigma_{k+2}\sigma_k}{\sigma_k^2}\\
\\
\ge& \frac{k+1}{n-k}F^2=\frac{1}{c_{n,k}}F^2,
      \end{array}
      \label{}
    \end{equation}
where the last inequality follows from Newton-McLaurin inequality.

Similarly, we have
    \begin{equation}
      \begin{array}[]{rll}
       \sum  F^{ii} =&\sum \frac{\sigma_{k}(\lambda|i)\sigma_k(\lambda) -\sigma_{k+1}(\lambda)\sigma_{k-1}(\lambda|i)}{\sigma_k^2}\\
\\
=&\frac{(n-k)\sigma^2_{k} -(n-k+1)\sigma_{k+1}\sigma_{k-1}}{\sigma_k^2}\\
\\
\ge& \frac{n-k}{k+1}=c_{n,k},
      \end{array}
      \label{u}
    \end{equation}
where the last inequality follows from Newton-McLaurin inequality. If $\lambda\in \bar \Gamma_{k+1}$, then $\sigma_{k+1}\sigma_{k-1}\ge 0$ and we conclude that $\sum F^{ii} \leq n-k$ from the second identity in (\ref{u}).
\end{proof}

\begin{theorem}\label{th4.3}
Let $M_0$ be a strictly convex, radial graph of positive function $\rho_0$ over $\mathbb{S}^n$ embedded in $\mathbb{S}^{n+1}$. If $M(t)$ solves the flow \eqref{1.3} with the initial value $M_0$, then for any $(z,t) \in \mathbb{S}^n \times [0, T)$
\begin{align} \label{4.7}
u(z,t) \geq \mathop {\min }\limits_{z \in \mathbb{S}^n} u(z, 0).
\end{align}
As a consequence, we have $C^1$ bound for $\rho$, that is,
\begin{align}
|\rho|_{C^1(\mathbb{S}^n)} \leq C,
\end{align}
where $C$ depends only on the initial data.
\end{theorem}

\begin{proof}
For any fixed $t \in (0, T)$, we have at the minimum point of $u(z,t)$,
\[
\nabla u =0.
\]
So from Lemma \ref{lem2.11}, we have
\begin{align}
\partial _t u - uF^{ij} u_{ij}  =& c_{n,k} |\nabla \Phi |^2 + (c_{n,k} \phi ' - 2uF)\phi ' + u^2 F^{ij} (h^2 )_{ij}\notag \\
\geq& c_{n,k} |\nabla \Phi |^2 + c_{n,k}( \phi ' - \frac{1}{c_{n,k}}uF ) ^2\geq 0,
\end{align}
which finishes the proof of \eqref{4.7}.
\end{proof}

\subsection{Uniform bounds of $F$}
\begin{lemma}\label{lem4.4}
Let $\lambda =(\lambda_1, \cdots, \lambda_n) \in \Gamma_{n}$, and $\lambda_1 \geq \cdots \geq \lambda_n$. Then for $1 \leq m \leq n-1$, we have
\begin{align}\label{4.10}
\frac{{m(n - m)\sigma _m (\lambda )^2  - (m + 1)(n - m + 1)\sigma _{m + 1} (\lambda )\sigma _{m - 1} (\lambda )}}{{\sigma _m (\lambda )^2 }} \sim \frac{(\lambda_1-\lambda_n)^2}{{\lambda_1}^2},
\end{align}
where $f \sim g$ means $ \frac{1}{C(n,m)} g \leq f \leq C(n,m) g$ for some constant $C(n,m) >0$.
\end{lemma}

\begin{proof}
The case $m=1$ is trivial, as \[(n-1)\sigma_1(\lambda)^2-2n\sigma_2(\lambda)=\sum_{i<j}(\lambda_i-\lambda_j)^2.\] We may assume $1<1<n$.
By direct computation we can derive
\begin{align}
m&(n - m)\sigma _m (\lambda )^2  - (m + 1)(n - m + 1)\sigma _{m + 1} (\lambda )\sigma _{m - 1} (\lambda ) \notag \\
=& [m\sigma _m (\lambda )][(n - m)\sigma _m (\lambda )] - [(m + 1)\sigma _{m + 1} (\lambda )][(n - m + 1)\sigma _{m - 1} (\lambda )] \notag\\
=& [\sum\limits_i {\lambda _i \sigma _{m - 1} (\lambda |i)} ][\sum\limits_j {\sigma _m (\lambda |j)} ] - [\sum\limits_j {\lambda _j \sigma _m (\lambda |j)} ][\sum\limits_i {\sigma _{m - 1} (\lambda |i)} ]\notag \\
=& \sum\limits_{i,j} {(\lambda _i  - \lambda _j )\sigma _{m - 1} (\lambda |i)\sigma _m (\lambda |j)} \notag \\
=& \sum\limits_{i < j} {(\lambda _i  - \lambda _j )[\sigma _{m - 1} (\lambda |i)\sigma _m (\lambda |j) - \sigma _{m - 1} (\lambda |j)\sigma _m (\lambda |i)]}  \notag\\
=& \sum\limits_{i < j} {(\lambda _i  - \lambda _j )[\left( {\sigma _{m - 1} (\lambda |ij) + \lambda _j \sigma _{m - 2} (\lambda |ij)} \right)\left( {\sigma _m (\lambda |ij) + \lambda _i \sigma _{m - 1} (\lambda |ij)} \right)} \notag \\
&- \left( {\sigma _{m - 1} (\lambda |ij) + \lambda _i \sigma _{m - 2} (\lambda |ij)} \right)\left( {\sigma _m (\lambda |ij) + \lambda _j \sigma _{m - 1} (\lambda |ij)} \right)] \notag\\
=& \sum\limits_{i < j} {(\lambda _i  - \lambda _j ) \cdot (\lambda _i  - \lambda _j )[\sigma _{m - 1} (\lambda |ij)^2  - \sigma _{m - 2} (\lambda |ij)\sigma _m (\lambda |ij)]}  \notag\\
\sim& \sum\limits_{i < j} {(\lambda _i  - \lambda _j )^2 \sigma _{m - 1} (\lambda |ij)^2 },
\end{align}
where the last $\sim$ follows from Newton-MacLaurin inequality. Then we can obtain
\begin{align}
&\frac{{m(n - m)\sigma _m (\lambda )^2  - (m + 1)(n - m + 1)\sigma _{m + 1} (\lambda )\sigma _{m - 1} (\lambda )}}{{\sigma _m (\lambda )^2 }} \notag \\
\sim& \sum\limits_{i < j} {(\lambda _i  - \lambda _j )^2 \frac{{\sigma _{m - 1} (\lambda |ij)^2 }}{{\sigma _m (\lambda )^2 }}}  \notag\\
\sim& \sum\limits_{m \le i < j} {(\lambda _i  - \lambda _j )^2 \frac{{(\lambda _1  \cdots \lambda _{m - 1} )^2 }}{{(\lambda _1  \cdots \lambda _m )^2 }}}  + \sum\limits_{i < m < j} {(\lambda _i  - \lambda _j )^2 \frac{{(\lambda _1  \cdots \widehat{\lambda _i } \cdots \lambda _m )^2 }}{{(\lambda _1  \cdots \lambda _m )^2 }}}  \notag\\
&+ \sum\limits_{i < j \le m} {(\lambda _i  - \lambda _j )^2 \frac{{(\lambda _1  \cdots \widehat{\lambda _i } \cdots \widehat{\lambda _j } \cdots \lambda _{m + 1} )^2 }}{{(\lambda _1  \cdots \lambda _m )^2 }}}  \notag\\
\sim& \sum\limits_{m \le i < j} {\frac{{(\lambda _i  - \lambda _j )^2 }}{{\lambda _m ^2 }}}  + \sum\limits_{i < m < j} {\frac{{(\lambda _i  - \lambda _j )^2 }}{{\lambda _i ^2 }}}  + \sum\limits_{i < j \le m} {\frac{{(\lambda _i  - \lambda _j )^2 \lambda _{m + 1} ^2 }}{{(\lambda _i \lambda _j )^2 }}} \notag \\
\sim& \frac{{(\lambda _m  - \lambda _n )^2 }}{{\lambda _m ^2 }} + \frac{{(\lambda _1  - \lambda _n )^2 }}{{\lambda _1 ^2 }} + \frac{{(\lambda _1  - \lambda _m )^2 }}{{\lambda _1 ^2 }}\frac{{\lambda _{m + 1} ^2 }}{{\lambda _m ^2 }} \notag\\
\sim& \frac{{(\lambda _1  - \lambda _n )^2 }}{{\lambda _1 ^2 }},
\end{align}
where $\widehat{\lambda _i }$ means that $\lambda _i$ is omitted. The proof is finished.
\end{proof}

Due to Lemma \ref{lem4.4}, we obtain the uniform bound of $F$ as follows.
\begin{theorem}\label{th4.5}
Let $M_0$ be a strictly convex, radial graph of positive function $\rho_0$ over $\mathbb{S}^n$ embedded in $\mathbb{S}^{n+1}$. If $M(t)$ solves the flow \eqref{1.3} with the initial value $M_0$, then for any $(z,t) \in \mathbb{S}^n \times [0, T)$
\begin{align} \label{4.13}
\frac{1}{C} \leq F \leq C,
\end{align}
where $C$ depends only on $n$, $k$ and the initial data.
\end{theorem}

\begin{proof}
For any fixed $t \in (0, T)$, we have at the critical points of $F$,
\[
\nabla F =0.
\]
So from Lemma \ref{lem2.13}, we can get
\begin{align} \label{4.14}
 \partial _t F - uF^{ij} \nabla _i \nabla _j F =& - c_{n,k} \phi 'F^2 [ \frac{F^{ij} (h^2 )_{ij}}{F^2}  - \frac{1}{c_{n,k}}] +  uF [\sum {F^{ii} }  - c_{n,k} ].
\end{align}
In the following, we divide the proof of \eqref{4.13} into three cases.

Firstly, for the Case: $1 \leq k \leq n-2$, we know from Lemma \ref{lem4.4},
\begin{align}
 \frac{{F^{ij} (h^2 )_{ij} }}{{F^2 }} - \frac{1}{{c_{n,k} }} = \frac{{(k + 1)\frac{{n - k - 1}}{{n - k}}\sigma _{k + 1} ^2  - (k + 2)\sigma _{k + 2} \sigma _k }}{{\sigma _{k + 1} ^2 }} \sim \frac{{(\lambda_1  - \lambda_n )^2 }}{{\lambda_1 ^2 }},
 \end{align}
 and
\begin{align}
 \sum {F^{ii} }  - c_{n,k}  = \frac{{\frac{k}{{k + 1}}(n - k)\sigma _k ^2  - (n - k + 1)\sigma _{k + 1} \sigma _{k - 1} }}{{\sigma _k ^2 }} \sim \frac{{(\lambda_1  - \lambda_n )^2 }}{{\lambda_1 ^2 }}.
\end{align}
Thus
\begin{align}
 \frac{{F^{ij} (h^2 )_{ij} }}{{F^2 }} - \frac{1}{{c_{n,k} }} \sim \sum {F^{ii} }  - c_{n,k}.
\end{align}
Hence \eqref{4.13} holds.

Secondly, for the Case: $k=0$, we can directly get,
\begin{align}
 \frac{{F^{ij} (h^2 )_{ij} }}{{F^2 }} - \frac{1}{{c_{n,k} }} \geq 0,
 \end{align}
 and
\begin{align}
 \sum {F^{ii} }  - c_{n,k}  = 0,
            \end{align}
so we can get $ F \leq C$ from \eqref{4.14}. To prove the lower bound of $F$, we consider the minimum point of $uF$, and then we can get $ F \geq \frac{1}{C}$ from \eqref{2.38}. Hence \eqref{4.13} holds.

Lastly, we consider the Case: $k= n-1$ in the following. It is easy to know
\begin{align}
 \frac{{F^{ij} (h^2 )_{ij} }}{{F^2 }} - \frac{1}{{c_{n,k} }} =0,
 \end{align}
 and
\begin{align}
(n-k) - c_{n,k} \geq \sum {F^{ii} }  - c_{n,k} \geq 0,
\end{align}
so we can get $ F \geq \frac{1}{C}$ from \eqref{4.14}.

To prove the upper bound of $F$, we consider the maximum point of $P =: F{\rm{ + }}\frac{1}{u}$. At the maximum point of $P$, we have
\begin{align}
 0 = \nabla _i P = \nabla _i F - \frac{1}{{u^2 }}\nabla _i u,
\end{align}
and then
\begin{align} \label{4.23}
 \partial _t P - uF^{ij} \nabla _i \nabla _j P =& \partial _t F - uF^{ij} \nabla _i \nabla _j F \notag \\
  &- \frac{1}{{u^2 }}\left[ {\partial _t u - uF^{ij} \nabla _i \nabla _j u} \right] - uF^{ij}  \cdot \frac{2}{{u^3 }}\nabla _i u\nabla _j u \notag \\
  =& 2F^{ij} \nabla _i u\nabla _j F + F\nabla \Phi \nabla F + uF[\sum {F^{ii} }  - c_{n,k} ] \notag \\
  &- \frac{1}{{u^2 }}\left[ {-c_{n,k} |\nabla \Phi |^2  + F\nabla \Phi \nabla u + (c_{n,k} \phi ' - 2uF)\phi ' + u^2 F^{ij} (h^2 )_{ij} } \right] \notag \\
  &- uF^{ij}  \cdot \frac{2}{{u^3 }}\nabla _i u\nabla _j u \notag \\
  =& uF[\sum {F^{ii} }  - c_{n,k} ]- \frac{1}{{u^2 }}\left[ -c_{n,k} |\nabla \Phi |^2  + (c_{n,k} \phi ' - 2uF)\phi ' + u^2 \frac{F^2}{c_{n,k}} \right].
  \end{align}
So we can get $ F \leq C$  from \eqref{4.23}. Hence \eqref{4.13} holds.

\end{proof}

\section{Preserving convexity}

In this section, we prove the flow \eqref{1.3} preserves convexity in $\mathbb{S}^{n+1}$. Denote $T>0$ to be largest time, up to which all flow hypersurfaces are strictly convex. Furthermore denote by $T^{*}$ is the largest time of existence of a smooth solution to (\ref{1.3}).

\begin{theorem} \label{th3.1}
Let $M(t)$ be an oriented immersed connected hypersurface in $\mathbb{S}^{n+1}$ with a positive semi-definite second
fundamental form $h(t) \in \Gamma_{k+1}$ satisfying equation \eqref{1.3} for $t \in [0, T^*)$, then then $M(t)$ is strictly convex for all
$t \in (0, T^*)$.
\end{theorem}

Here we provide two proofs. The first proof follows from the following lemma.

\begin{lemma}\label{pres-conv}
Along the solution of \eqref{1.3} with a strictly convex initial hypersurface $M_{0}\subset\mathbb S^{n+1}$ all flow hypersurfaces $M_{t}=X(M,t)$ are strictly convex up to $T^{*}$, i.e. $T=T^{*}$, with a uniform estimate
\[h^{i}_{j}\geq c \delta^{i}_{j},\]
where $c=c(\sup_{M_{0}}\rho,\inf_{M_{0}}\rho,n,k,T^{*})$.
\end{lemma}

\begin{proof}
We calculate the evolution equation of the inverse $\{b ^i_{j}\}= \{h_i^j\}^{-1}$, which is well defined up to $T$. We suppose that $T<T^{*}$.
\begin{align*}
\partial_t {b}^m_m=&-b^m_r \partial_t {h}^r_s b^s_m,\\
F^{ij}\nabla_i\nabla_j b^m_{m}=&2 F^{ij}b^m_r\nabla_i h^r_sb^s_p \nabla_j h^p_q b^q_m-F^{ij}b^m_r \nabla_i\nabla_j h^r_s b^s_m,\\
\nabla_i u =& h_{i}^j \nabla_j \Phi,
\end{align*}
and by evolution equation for $h_{i}^j$, we deduce
\begin{align}\label{pres-conv-1}
&\partial_t {b}^m_m-uF^{ij}\nabla_i\nabla_j b^m_{m} - F \nabla^i\Phi  \nabla_i b^m_{m}\notag \\
=&-u(F^{pq,rs}+2 F^{qs}b^{pr} )\nabla_i h_{pq} \nabla^j h_{rs} b^m_jb^i_m- b^{m}_{j}\nabla^{j}{F}\nabla_m \Phi- b^{i}_{m} \nabla_i F \nabla^m{\Phi}  \nonumber  \\
&-uF^{ij}h_{ir}h^{r}_{j}b^{m}_{m}+(c_{n,k}\phi'+uF)-\phi'Fb^{m}_{m}\\ \nonumber
&+uF^{ij}g_{ij}b^{m}_{m}+c_{n,k}ub^{m}_{m}-2uFb^{m}_{r}b^{r}_{m}\\ \nonumber
&\leq -\frac{2u}{F} \nabla_i F \nabla^j {F} b^{m}_{j}b^{i}_{m}- b^{m}_{j} \nabla^j {F} \nabla_m \Phi - b^{i}_{m} \nabla_i F \nabla^m {\Phi}\\ \nonumber
&+\psi_{1}(t)b^{m}_{m}+\psi_{2}(t)-2uFb^{m}_{r}b^{r}_{m},
\end{align}
where we used the inverse concavity of $F$, cf. Theorem 2.3 in \cite{Andrews:/2007} and where $\psi_{i}$ are smooth functions which are uniformly bounded up to $T$, due to the uniform upper and lower bounds of $F$.

We use a well known trick to estimate the maximal eigenvalue of $b$, e.g. compare Lemma~6.1 \cite{Gerhardt:01/1996}.
Let
\[Q=\sup\{b_{ij}\eta^i\eta^j\ |\quad  g_{ij}\eta^i\eta^j=1\},\]
and suppose this function attains a maximum at $(t_0,\xi_0)$ with $t_0<T,$ i.e.
\[Q(t_{0},\xi_{0})=\sup_{[0,t_{0}]\times M}Q.\]
 Choose coordinates in $(t_0,\xi_0)$ with
\[g_{ij}=\delta_{ij},\quad b_{ij}=\lambda_i^{-1}\delta_{ij},\quad \lambda_1^{-1}\leq \dots\leq \lambda_n^{-1}.\]
Let $\eta$ be the vector field $\eta=(0,\dots,0,1)$ and define
\[\tilde
Q=\frac{b_{ij}\eta^i\eta^j}{g_{ij}\eta^i\eta^j},\]
then locally around $(t_0,\xi_0)$ we have $\tilde Q\leq Q$ and the derivatives coincide. Thus at $(t_0,\xi_0)$ the function $\tilde Q$ and $b^n_n$ satisfy the same evolution equation and we may show that the right hand side of \eqref{pres-conv-1} is negative at the point $(t_0,\xi_0)$ in these coordinates, yielding a contradiction.

In these coordinates we obtain
\begin{align}\label{pres-conv-2}
&\partial_t {b}^n_n-uF^{ij}\nabla_i\nabla_j b^n_{n} - F \nabla^i\Phi  \nabla_i b^n_{n} \notag\\
\leq &-\frac{2u}{F}(\nabla_n F)^{2}\lambda_{n}^{-2}-2 \lambda_{n}^{-1}{\nabla_n F}\nabla_n \Phi+\psi_{1}(t)\lambda_{n}^{-1}+\psi_{2}(t)-2uF\lambda_{n}^{-2}\nonumber \\
\leq& -\frac{2u}{F}(\nabla_n F)^{2}\lambda_{n}^{-2}+\epsilon\lambda_{n}^{-2}(\nabla_n F)^{2}+c_{\epsilon}(\nabla_n \Phi)^{2}+\psi_{1}(t)\lambda_{n}^{-1}+\psi_{2}(t)-2uF\lambda_{n}^{-2} \nonumber\\
<&0,
\end{align}
for small $\epsilon$ and large $\lambda_{n}^{-2}$. Hence we obtain that $\lambda_{n}^{-1}$ does {\it{not}} blow up at $T$, in contradiction to the definition of $T<T^{*}$. Hence we must have $T=T^{*}$, with a uniform lower bound on $h_{ij}$ on finite intervals. The proof is complete.  \end{proof}

The second proof is from the Constant Rank Theorem in Bian-Guan \cite{BG09} , along the lines of proof of Theorem 6.1 in \cite{GL15} (where the constant rank theorem was proved for a general flow in $\mathbb R^{n+1}$). For (\ref{1.3}) in $\mathbb S^{n+1}$, we have an extra good term which is associated to the curvature $K=1$. We outline the arguments here with necessary modification.

\begin{proof}

Let $W = (g^{im} h_{mj} )$, and $l(t)$ be the minimal rank of
$W$. Suppose at $W$ is degenerate $(x_0, T)$, such that $W(T)$ attains minimal rank $l<n$
at $x_0$. Set \[\varphi(x, t) = \sigma_{l+1}(W(x, t)) +\frac{\sigma_{l+2}(W(x, t))}{\sigma_{l+1}(W(x, t))}.\] It is proved in section
2 in \cite{BG09} that $\varphi$ is in $C^{1,1}$.

As in Bian-Guan \cite{BG09}, near $(x_0, T)$, the index set $\{1, 2, \cdots, n\}$ can be divided in to two subsets $B,G$, where for $i \in B$, the eigenvalues of $\{ W_{ij} \}$, $\lambda_i$ is small and for
$j \in G$, $\lambda_j$ is strictly positive away from 0. As in \cite{BG09}, we may assume at each
point of computation, $\{ W_{ij} \}$ is diagonal. Notice that $W_{ii} \leq C\varphi$ for all $i \in B$.

Denote $G = uF - c_{n,k} \phi '$ and $F = \frac{{\sigma _{k + 1} (\lambda )}}{{\sigma _k (\lambda)}}$. From \eqref{EV-W} we recall
\begin{align}\label{3.1}
\partial _t h_i^i  = \nabla ^i \nabla _i G + Gg^{ik} (h^2 )_{ki}  + G.
\end{align}
So we have the following equality
\begin{align}
 &\sum {G^{\alpha \beta } \varphi _{\alpha \beta } }  - \varphi _t\notag\\
   = &~O(\varphi  + \sum\limits_{i,j \in B} {|\nabla W_{ij} |} ) - \frac{1}{{\sigma _1 (B)}}\sum\limits_{\alpha \beta } {\sum\limits_{i \ne j \in B} {G^{\alpha \beta } W_{ij,\alpha } W_{ij,\beta } } } \notag \\
  &- \frac{1}{{\sigma _1 (B)^3 }}\sum\limits_{\alpha \beta } {\sum\limits_{i \in B} {G^{\alpha \beta } (W_{ii,\alpha } \sigma _1 (B) - W_{ii} \sum\limits_{j \in B} {W_{jj,\alpha } } )(W_{ii,\beta } \sigma _1 (B) - W_{ii} \sum\limits_{j \in B} {W_{jj,\beta } } )} }  \notag \\
  &- 2\sum\limits_{i \in B} {[\sigma _l (G) + \frac{{\sigma _1 (B|i)^2  - \sigma _2 (B|i)}}{{\sigma _1 (B)^2 }}]\sum\limits_{\alpha \beta } {\sum\limits_{j \in G} {G^{\alpha \beta } \frac{{W_{ij,\alpha } W_{ij,\beta } }}{{W_{jj} }}} } } \notag  \\
  &+ \sum\limits_{i \in B} {[\sigma _l (G) + \frac{{\sigma _1 (B|i)^2  - \sigma _2 (B|i)}}{{\sigma _1 (B)^2 }}]\left( {\sum\limits_{\alpha \beta } {G^{\alpha \beta } W_{ii,\alpha \beta } }  - \partial _t W_{ii} } \right)}
  \end{align}
By \eqref{3.1} and (\ref{2.36}),
\begin{align}
 \partial _t h_i^i  = u[F^{\alpha \beta } \nabla_\alpha  \nabla _\beta  h_{ii}  + F^{\alpha \beta ,\gamma \eta }\nabla _i h_{\alpha\beta }  \nabla _i h_{\gamma\eta } ]+ O(\varphi  + \sum\limits_{i,j \in B} {|\nabla W_{ij} |} ) + 2uF,
  \end{align}
where we used the facts $\nabla_i u = O(\varphi)$, $\nabla_i \nabla_i u = O(\varphi + |\nabla W_{ii} |)$ and $\nabla_i \nabla_i \Phi = \phi' +O(\varphi )$, $\forall i\in B$.

As $
G^{\alpha \beta }  = uF^{\alpha \beta }$,
\begin{align}
\sum\limits_{\alpha \beta } {G^{\alpha \beta } W_{ii,\alpha \beta } }  - \partial _t W_{ii}  =& u \sum\limits_{\alpha \beta } {F^{\alpha \beta } \nabla _\alpha  \nabla _\beta  h_{ii} }  - \partial _t h_i^i  \notag \\
=&O(\varphi  + \sum\limits_{i,j \in B} {|\nabla W_{ij} |} ) -u F^{\alpha \beta ,\gamma \eta }W_{\alpha\beta,i }  W_{\gamma\eta,i } - 2uF,
\end{align}
Since $F$ satisfies the structure condition in \cite{BG09}
\[F^{\alpha \beta ,\gamma \eta }W_{\alpha\beta,i }  W_{\gamma\eta,i } +2\sum\limits_{\alpha \beta } \sum\limits_{j \in G} F^{\alpha \beta } \frac{{W_{ij,\alpha } W_{ij,\beta } }}{{W_{jj} }}\ge 0.\]
We obtain
\begin{align}
 \sum {G^{\alpha \beta } \varphi _{\alpha \beta } }  - \varphi _t    \leq C(\varphi + |\nabla \varphi|) -C \sum\limits_{i,j \in B} {|\nabla W_{ij} |}
 -2uF \sum\limits_{i \in B} {[\sigma _l (G) + \frac{{\sigma _1 (B|i)^2  - \sigma _2 (B|i)}}{{\sigma _1 (B)^2 }}]}. \notag
\end{align}
Following the analysis in the proof of Theorem 3.2 in \cite{BG09}, it yields
\begin{align}
 \sum {G^{\alpha \beta } \varphi _{\alpha \beta } }  - \varphi _t    \leq C(\varphi + |\nabla \varphi|)
 -2uF\sigma _l (G). \notag
\end{align}
This is a contradiction from the standard strong maximum principle for parabolic equations.
\end{proof}

\bigskip

\begin{remark}
The preservation convexity of flow (\ref{1.3}) when $k=0$ was first observed by JS and Chao Xia \cite{JX2019}.
\end{remark}

Since in the case of $k=0$, the longtime existence and convergence was proved in \cite{GL15}, the following sharp inequality follows.
\begin{proposition}
If $\Omega\subset \mathbb S^{n+1}$ is convex, then
\[
\mathrm{Vol}(\Omega)\le  \xi_{l,-1}(\mathcal A_l) \quad \forall l\ge 0,
\]
 with equality holds iff $\Omega $ is a convex geodesic ball.
\end{proposition}

\section{Discussion of $C^2$ estimate}

The curvature estimate for flow (\ref{1.3}) is still open. In the case of $\mathbb R^{n+1}$, $\phi'=1$, the corresponding flow  is
\begin{equation}\label{1.3R}
X_t=(c_{n,k}-uF(\lambda))\nu.
\end{equation}
Flow (\ref{1.3R}) has the same curvature estimate issue. In \cite{GL18, GL 20} the flow (\ref{1.3R}) was converted to a corresponding inverse type flow for the Euclidean support function $u$ of the evolving convex body parametrized on the outer normals (i.e. on $\mathbb S^n$). The longtime existence and convergence of the admissible solution were proved in \cite{GL18, GL 20}.

Below we will convert flow (\ref{1.3}) to an evolution of convex bodies in $\mathbb R^{n+1}$ and we write down the evolution equation of the corresponding Euclidean support function $\tilde u$ on $\mathbb S^n$.

Introduce a new variable $\gamma$ satisfying
\begin{align} \label{5.2}
\frac{{d\gamma }}{{d\rho }} = \frac{1}{\phi }.
\end{align}
Let $\omega  = \sqrt {1 + |\nabla \gamma |^2 }$, one can compute the unit outward normal $\nu  = \frac{1}{\omega }(1, - \frac{{\nabla \rho }}{{\phi ^2 }})$, and
\begin{align}
\label{5.3} u =& \frac{\phi }{\omega }, \\
\label{5.4} g_{ij}  =& \phi ^2 e _{ij}  + \rho _i \rho _j , \\
\label{5.5} g^{ij}  =& \frac{1}{{\phi ^2 }}(e^{ij}  - \frac{{\gamma^{i} \gamma^{j} }}{{\omega ^2 }}), \\
\label{5.6} h_{ij}  =& \frac{\phi }{\omega }( - \gamma _{ij}  + \phi '\gamma _i \gamma _j  + \phi 'e_{ij} ) ,\\
\label{5.7} h_j^i  =& \frac{1}{{\phi \omega }}(e^{im}  - \frac{{\gamma^i \gamma^m }}{{\omega ^2 }})( - \gamma _{mj}  + \phi '\gamma _m \gamma _j  + \phi 'e_{mj} ).
\end{align}

It follows from \eqref{2.29} that the evolution equation for $\gamma$ is
 \begin{align} \label{5.8}
 \partial _t \gamma  =& \frac{1}{\phi }\partial _t \rho  = f\frac{\omega }{\phi } \notag \\
  =& c_{n,k} \frac{{\phi '}}{u} - \frac{{\sigma _{k + 1} (\lambda )}}{{\sigma _k (\lambda )}}.
\end{align}
From \eqref{5.7},
\begin{align}
h_j^i  =& \frac{1}{{\phi \omega }}(e^{im}  - \frac{{\gamma^i \gamma^m }}{{\omega ^2 }})( - \gamma _{mj}  + \phi '\gamma _m \gamma _j  + \phi 'e_{mj} ) \notag  \\
 =& \frac{1}{{\phi \omega }}\big\{(e^{im}  - \frac{{\gamma^i \gamma^m }}{{\omega ^2 }})( - \gamma _{mj}  +\gamma _m \gamma _j  + e_{mj} ) + (\phi '-1) \delta^i_j \big\}\notag  \\
 =& \frac{e^\gamma}{{\phi }} \widetilde{h}_j^i + \frac{\phi '-1}{{\phi \omega }} \delta_{j}^i ,
\end{align}
where
$$\widetilde{h}_j^i = \frac{1}{{e^\gamma \omega }} \big\{(e^{im}  - \frac{{\gamma^i \gamma^m }}{{\omega ^2 }})( - \gamma _{mj}  +\gamma _m \gamma _j  + e_{mj} ),$$
 which is the Weingarten tensor of the graph $\widetilde{M}$ (over $\mathbb{S}^n$) of radial function $\widetilde{\rho} = e^\gamma$ in $\mathbb{R}^{n+1}$. Since $M(t)$ is strictly convex (i.e. $\{h_j^i\} >0$), and $\phi'-1 = \cos \rho -1 <0$,  thus $\{\widetilde{h}_j^i\} >0$. That is $\widetilde{M}$ is strictly convex.

We have
\begin{align} \label{5.10}
 \partial _t \widetilde{\rho} =& \widetilde{\rho} \partial _t \gamma  \notag \\
  =& c_{n,k} \frac{{\phi '}}{\phi} \widetilde{\rho} \omega - \frac{{\widetilde{\rho}^2}}{\phi} \frac{{\sigma _{k + 1} }}{{\sigma _k}}(\widetilde{h}_j^i +  \frac{\phi '-1}{{\widetilde{\rho} \omega }} \delta_{j}^i).
\end{align}
Let $\widetilde{u}$ be the support function of the strictly convex body $\widetilde{M}$, and $z$ and $\nu$ be the unit radial vector and the unit outer normal vector of $\widetilde{M}$, respectively. Then from $\widetilde{\rho}(z,t) (z \cdot \nu) = \widetilde{u}(\nu,t)$, we can get $\log \widetilde{\rho}(z,t) = \log \widetilde{u}(\nu,t) - \log (z \cdot \nu)$ and
\begin{align*}
\frac{{1}} {{ \widetilde{\rho}(z,t)}}\frac{{\partial \widetilde{\rho}(z,t)}} {{\partial t}} =& \frac{{1}} {{\widetilde{u}(\nu,t)}} [\nabla \widetilde{u} \cdot \nu_t + \widetilde{u}_t ] - \frac{z \cdot \nu_t}{z \cdot \nu}  \notag \\
=& \frac{{1}} {{\widetilde{u}(\nu,t)}} \frac{{\partial \widetilde{u}(\nu,t)}} {{\partial t}} + \frac{{1}} {{\widetilde{u}(\nu,t)}}[(\nabla \widetilde{u} -  \widetilde{\rho}(z,t) z)\cdot \nu_t]\notag \\
=& \frac{{1}} {{\widetilde{u}(\nu,t)}} \frac{{\partial \widetilde{u}(\nu,t)}} {{\partial t}}.
\end{align*}
Denote $W_{\widetilde{u}} =: \{ \widetilde{u}_{ij} + \widetilde{u} \delta_{ij} \} = \{\widetilde{h}_j^i\} ^{-1}$ and we can get the evolution equation of $\widetilde{u}$ as follows
\begin{align} \label{5.11}
 \partial _t \widetilde{u} =& \frac{ \widetilde{u}}{\widetilde{\rho}} \partial _t  \widetilde{\rho} \notag \\
  =& c_{n,k} \frac{{\phi '}}{\phi} \widetilde{u} \omega - \frac{{\widetilde{\rho} \widetilde{u} }}{\phi} \frac{{\sigma _{k + 1} }}{{\sigma _k}}(\widetilde{h}_j^i +  \frac{\phi '-1}{{\widetilde{\rho} \omega }} \delta_{j}^i) \notag \\
  =&\colon G(W_{\widetilde{u}}, \widetilde{u}, \nabla \widetilde{u} ).
\end{align}
This equation is of inverse type, and {\it the question is whether \eqref{5.11} exists for all time?}

\bigskip


\section*{Acknowledgments:}
Part of this work was done while CC was visiting McGill University in 2018, he would
like to thank McGill and PG for the warm hospitality.

This work was made possible through a research scholarship JS received from the DFG and which was carried out at Columbia University in New York. JS would like to thank the DFG, Columbia University and especially Prof. Simon Brendle for their support.

Parts of this work were written during a visit of JS to McGill University in Montreal. JS would like to thank McGill and PG for their hospitality and support.

\end{document}